\documentclass[12pt,tgrind,psbox]{article}

\usepackage{latexsym}
\usepackage{graphics}
\usepackage{amsmath}
\usepackage{xspace}
\usepackage{amssymb}
\usepackage{psfrag}
\usepackage{epsfig}
\usepackage{amsmath,amsthm}
\usepackage{epsfig}
\usepackage{pst-all}

\newcommand {\bel}[1]{\begin{align*}}
\newcommand {\eel}[1]{\end{align*}}
\newcommand {\bea}{\begin{eqnarray}}
\newcommand {\eea}{\end{eqnarray}}

\newcommand{\R}{\mathbb{R}}
\newcommand{\Z}{\mathbb{Z}}

\newcommand{\mb}[1]{\mbox{\boldmath $#1$}}

\newcommand{\ignore}[1]{\relax}

\newtheorem{theorem}{Theorem}

\newtheorem{lemma}{Lemma}

\newtheorem{prop}{Proposition}

\newtheorem{Defi}{Definition}

\definecolor{Red}{rgb}{1,0,0}
\definecolor{Blue}{rgb}{0,0,1}
\definecolor{Olive}{rgb}{0.41,0.55,0.13}
\definecolor{Green}{rgb}{0,1,0}
\definecolor{MGreen}{rgb}{0,0.8,0}
\definecolor{DGreen}{rgb}{0,0.55,0}
\definecolor{Yellow}{rgb}{1,1,0}
\definecolor{Cyan}{rgb}{0,1,1}
\definecolor{Magenta}{rgb}{1,0,1}
\definecolor{Orange}{rgb}{1,.5,0}
\definecolor{Violet}{rgb}{.5,0,.5}
\definecolor{Purple}{rgb}{.75,0,.25}
\definecolor{Brown}{rgb}{.75,.5,.25}
\definecolor{Grey}{rgb}{.5,.5,.5}
\definecolor{Pink}{rgb}{1,0,1}
\definecolor{DBrown}{rgb}{.5,.34,.16}
\definecolor{Black}{rgb}{0,0,0}

\title{Stability of Skorokhod problem is undecidable}

\author{
 {\sf David Gamarnik }
  \thanks{Operations Research Center and Sloan School of Management, MIT, Cambridge, MA,  02139, e-mail: {\tt
gamarnik@mit.edu}.   This work was supported by NSF grant CMMI-0726733}
\and
{\sf Dmitriy Katz} \thanks{T.J. Watson Research Center, IBM, Yorktown Heights, NY,  10598, e-mail: {\tt
dimdim@mit.edu}}
}

\date{\today}

\begin{document}

\maketitle

\begin{abstract}
Skorokhod problem arises in studying Reflected Brownian Motion (RBM) on an non-negative orthant,
specifically in the context of queueing networks in the heavy traffic regime.
One of the key problems is identifying conditions for stability of a Skorokhod problem, defined as the property that trajectories
are attracted to the origin.  The stability conditions are known in dimension up to three, but not for general dimensions.

In this paper we explain the fundamental difficulties encountered in trying to establish stability conditions
for general dimensions. We prove that stability of Skorokhod problem is an undecidable property when the starting state
is a part of the input. Namely, there does not exist an algorithm
(a constructive procedure) for identifying stable Skorokhod problem in general dimensions.
\end{abstract}




\section{Introduction}\label{section:introduction}
The Skorokhod problem was introduced by A.~Skorokhod to model  stochastic processes, typically  diffusion processes, which are constrained to take
values in a particular subset $\mathcal{X}$ of the Euclidian space $\R^d$. The constraints take form of augmenting the underlying unconstrained process (diffusion)
with an additional process (pushing) which is "active" only when the process is on the boundary of $\mathcal{X}$.
When the underlying process is a Brownian motion, the corresponding Skorokhod problem is often called Semi-Martingale Reflected Brownian Motion (SRBM).
Solving a Skorokhod
problem roughly refers to the problem of identifying the augmentation such that the modified process is well defined.
A particularly important application of the Skorokhod problem is in the theory of queueing networks, where the state space
$\mathcal{X}$ is usually a non-negative orthant $\R^d_+$,
the constrained process corresponds
to the vector of queues appropriately scaled, and the pushing process corresponds to processes
describing cumulative idling of servers~\cite{ChenYaoBook},\cite{harrison},\cite{WilliamsSurvey95}. In
this case one can naturally construct a certain reflection matrix $R$ such that the idling process $Y(t)$ impacts the vector of queues via the matrix $R$.
(See the next section for formal definitions of a Skorokhod problem, SRBM and other related notions.)

A key question arising in the context of Skorokhod problem is stability - the property that the constraned process is positive recurrent.
Among other things, the importance of stability stems from the fact that it implies the existence of an invariant probability measure
for the underlying stochastic process. In the context of queueing theory, one can use the invariant measure to obtain important insights into key
performance measures of the underlying queueing network, such as steady state queue lengths and waiting times distributions~\cite{GamarnikZeevi},\cite{BudhirajaLee}.
Unfortunately, stability of a Skorokhod problem turned out to be a difficult property to analyze. In the context of SRBM the stability becomes a property which
depends on the parameters of  underlying Brownian motion, namely the drift and the covariance matrix, and the reflection matrix $R$.
An important advance was achieved by Dupuis and Williams~\cite{dupuis_williams} who connected stability of SRBM with the stability of the associated
so-called fluid model. One can think of this fluid model as an SRBM with a deterministic (zero covariance matrix) Brownian motion. They showed that
an SRBM is stable if every path in the corresponding fluid model is attracted to the origin. There are two issues, however, associated with this important
result. First, as it was shown recently in Bramson~\cite{BramsonSRBMinstabilityFluid}, the converse of this result is not true when $d=6$.
See Theorem~\ref{theorem:counter-to-Skorokhod} below for the precise statement.
Second, it still
leaves open the question of identifying stability conditions for the corresponding fluid model. Some sufficient conditions are known for restricted
classes of the reflection matrix $R$~\cite{WilliamsSurvey95}. Additionally, a full characterization of stable fluid models
as well as the underlying SRBM for a 3-dimensional Skorokhod
problem was obtained in a series of papers: El Kharroubi, Ben Tahar and Yaacoubi~\cite{Kharroubi2000},
El Kharroubi, Ben Tahar and Yaacoubi~\cite{Kharroubi2002}, Bramson, Dai and Harrison~\cite{BramsonDaiHarrisonSRBM}.
However, as of now, the characterization of stable SRBM or stable fluid models of SRBM in general dimensions is not known.

In this work we explain the fundamental difficulties encountered in trying to establish stability conditions for a fluid model of an SRBM
for general dimensions. We prove that stability of a fluid model of an SRBM is an undecidable property when the starting state of a fluid path
is a part of the input (see the next section for formal definitions and the precise statement). Namely, there does not exist an algorithm
(a constructive procedure) which determines whether a reflection of a \emph{given} linear path is attracted to the origin.
We further conjecture that stability of a fluid model of an SRBM
remains undecidable when the starting state is not part of the input. Namely, when stability is defined
as the property that \emph{all} fluid paths are attracted to the origin. This is stability definition used in~\cite{dupuis_williams},
and in order to distinguish it from stability of a given fluid path, we call it \emph{global stability} of a fluid model in this paper.
Likewise, we conjecture that stability of an SRBM is an undecidable property in general dimension $d$.
Our result continues a stream of earlier works~\cite{gamarnik_decidability},\cite{gamarnik_decidability_LD},\cite{GamarnikRogozhnikovDecidability},
where stability of constrained random walks in $\Z_+^d$ and multiclass queueing networks operating
under certain classes of scheduling policies was shown to be undecidable.

The concept of undecidability was introduced in the classical works of Alan Turing in 1930's and it is one of the principal
tools for establishing limitations of certain decision problems. A good reference of decidability (computability) is~\cite{sipser}.
Thanks to the work of Turing we know that certain
decision problems do not admit an effective solution in a sense of existence of an algorithms to solve them.
We should note that undecidability property is not related to the speed of algorithms, or specifically, whether a polynomial
time algorithm exists for a given problem. If a problem is shown to be undecidable, it means it does not admit \emph{any}
algorithm to solve it, no matter how slow the running time is allowed to be.
There are many examples of undecidable problems, including Turing Halting Problem, Post Correspondence Problem, Conways' game of life
and many others. A recent article~\cite{Goodman-Strauss} gives a nice overview of known undecidable problems in mathematics,
as well as the connection of this concept with G{\"o}del's Incompleteness Theorem.
We should note, however, that very few undecidable problems are known in the context of stochastic processes, probabilistic
cellular automata being the only notable exception perhaps~\cite{Goodman-Strauss}. Thus we believe that this article
along with~\cite{gamarnik_decidability} and \cite{GamarnikKatz} contributes to awareness of this important notion
in the community of researchers working in the area of stochastic processes.

Typically one establishes undecidability of a given problem by taking a problem which is already known to be undecidable,
and establishing a reduction from this problem to the underlying problem of interest. This is known as the \emph{reduction method}.
Recently several problems were proven to be undecidable in the area of control theory~\cite{blondel},\cite{blondel_survey},
\cite{BlondelTsitsiklisMatrixPair}. In particular the work of Blondel et al. \cite{blondel} used a device
known as \emph{counter machine} or \emph{counter automata} as a reduction tool.
In the present paper as in \cite{blondel} as well as in \cite{gamarnik_decidability}, and~\cite{GamarnikKatz}
our proof technique is also based on a reduction from a Counter Machine model.
In particular we state a known undecidable problem, namely the Halting Problem for a Counter Machine~\cite{hopcroft},
and then build a reduction from the Halting Problem of a Counter Machine into the stability problem of a fluid model
of a Skorokhod problem. If there was an algorithm to determine stability of a fluid model, it would imply
the existence of an algorithm for solving the Halting Problem and this would be a contradiction.

The remainder of the paper is organized as follows. In the next section we define the Skorokhod problem, SRBM, its fluid model and stability.
Our main result, Theorem~\ref{theorem:mainresult} is also stated in this section. Section~\ref{section:CounterAut} presents
Counter Machine and the Halting Problem which is used as a basis of our reduction. The proof of the main result is in
Sections~\ref{section:Reduction},\,\ref{section:dynamics} and \ref{section:wrapup}. Specifically, a reduction from a Counter Machine
to a Skorokhod problem is constructed in Section~\ref{section:Reduction}. In Section~\ref{section:dynamics} we show that the constructed
Skorokhod problem has dynamics which mimics the one of the underlying Counter Machine. Finally, in Section~\ref{section:wrapup} we
construct a  modification of the Skorokhod problem to connect the halting property of the underlying Counter Machine with stability of the Skorokhod problem.

We close this section with some notational conventions. $C([0,\infty),\R^d)$ denotes the space of continuous $\R^d$ valued functions defined
on $[0,\infty)$. $\mb{1}\{\cdot\}$ denotes the indicator function. $\delta(\cdot)$ denotes the Kronecker function. Namely, $\delta(x)=1$ for $x=0$
and $\delta(x)=0$ for any other real value $x$. All vectors are assumed to be column vectors. We use $A^T$ to denote a transposition of a matrix $A$.

\section{Skorokhod problem and stability}
Given a $d$-dimensional square matrix $R$, the Skorokhod problem is the problem of constructing a map
$\Psi: C([0,\infty),\R^d)\rightarrow C^2([0,\infty),\R^d_+)$, such that for every $x\in C([0,\infty),\R^d)$
the image $(y,z)=\Psi(x)$ satisfies the following properties
\begin{align}
z(t)&=x(t)+Ry(t), ~~t\in \R_+ \label{eq:z=x+Ry}\\
y(0)&=0, ~y_j(t)~\text{is non-decreasing for all}~j=1,\ldots,d \label{eq:yincreasing}\\
\int_0^\infty z_j(s)dy_j(s)&=0,~~j=1,\ldots,d \label{eq:non-idling}.
\end{align}
where the integral in (\ref{eq:non-idling}) is in Stieltjes  sense, which is well defined since $y_j$ are non-decreasing.
Intuitively, the meaning of the constraint (\ref{eq:non-idling}) is that the process $y_j$ can increase only at times when
$z_j=0$. We say that $y_j$ is \emph{active} in the time interval $(s_1,s_2)$ if $y_j$ is strictly increasing in this interval.
We also say that $y_j$ is active at the unit rate if it increases at the unit rate. Namely $y_j(s)-y_j(s_1)=s-s_1$ for all $s\in (s_1,s_2)$.
We say that $y_j$ is passive over $(s_1,s_2)$ if $y_j(s_2)=y_j(s_1)$.
The specific Skorokhod problem we will construct in this paper will have many variables active at the unit rate over various time intervals.

The existence of such a  $\Psi$ is completely and nicely characterized by the matrix property known as completely-$\mathcal{S}$ property.
A matrix $R$ is defined to be $\mathcal{S}$-matrix if there exists a $d$-vector $w\ge 0$, such that $Rw>0$ in a coordinate-wise sense.
A matrix $R$ is completely-$\mathcal{S}$ if every principal submatrix of $R$ is a $\mathcal{S}$-matrix. (A principal submatrix is the one where
the indices of rows and columns are the same). In particular every completely-$\mathcal{S}$ has positive diagonal elements. The following result
is established in a series of papers~\cite{BernardKharroubi89},\cite{BernardKharroubi91},\cite{TaylorWilliams}.

\begin{theorem}
A Skorokhod map $\Psi$ exists if and only if the matrix $R$ is completely-$\mathcal{S}$.
\end{theorem}

Two special cases of Skorokhod problem are of particular importance. The first one corresponds to the case when $x(t)$
is a stochastic process, specifically a Brownian motion with a starting state $z_0\in\R_+^d$,
drift vector $\theta\in\R^d$ and covariance matrix $\Sigma$,
typically assumed to be non-singular. This special case is usually called Semi-Martingale Reflected Brownian Motion (SRBM) and arises in the context of heavy
traffic theory of queueing networks~\cite{WilliamsSurvey95},\cite{ChenYaoBook}. An SRBM is thus completely specified
by data $(z_0,\theta,\Sigma,R)$. It shown in Taylor and Williams~\cite{TaylorWilliams} that, though the Skorokhod mapping $\Psi$
may not be unique for some completely-$\mathcal{S}$ matrices $R$, in the context of SRBM it is unique in law. Namely, the distribution
of $x(t),y(t),z(t)$ is uniquely defined by $(z_0,\theta,\Sigma,R)$.  An SRBM is defined to be stable if it is positive recurrent, in which case
there exists a unique time invariant distribution. The stability property does not depend on the starting state $z_0$, and thus is a property
of the triplet $(\theta,\Sigma,R)$. A key outstanding open problem is determining when is a triplet $(\theta,\Sigma,R)$ stable.
A significant partial progress is obtained by considering the so-called fluid models or fluid paths in a Skorokhod problem, and this is
our second important special case.

Given a vector $z_0\in \R_+^d$ and a vector $\theta\in\R^d$, consider  the  linear function  $x(t)=z_0+\theta t$,
and the corresponding (set of) solution(s) $(y(t),z(t))=\Psi(x(t))$. The triplet $(x(t),y(t),z(t))$ is called called a fluid path
for the reason discussed below. The system $(z_0,\theta,R)$ or $(\theta,R)$ is called a fluid model of an SRBM.
One can think of a fluid model as an SRBM with a degenerate (deterministic) Brownian motion input function $x(t)$.

\begin{Defi}
A fluid model $(z_0,\theta,R)$ is defined to be stable if every solution $(y(t),z(t))=\Psi(x(t))$ of the Skorokhod problem with
$x(t)=z_0+\theta t$ has property $\lim_{t\rightarrow\infty}z(t)=0$. A fluid model $(\theta,R)$ is defined to be
globally stable if it is stable for every starting state $z_0$.
\end{Defi}
The importance of this definition stems from the following result established in Dupuis and Williams~\cite{dupuis_williams}.
\begin{theorem}
Suppose $(\theta,R)$ is globally stable. Then an SRBM $(\theta,\Sigma,R)$ is stable for every non-singular covariance matrix $\Sigma$.
\end{theorem}
The proof of this result is based on the fluid rescaling (functional law of large numbers) technique and thus justifies
the terms "fluid model" and "fluid paths". One would naturally hope for a converse result, thus showing that stability
of SRBM is completely determined by the stability of its fluid model $(\theta,R)$. Unfortunately, this hope did not materialize
for six-dimensional SRBMs.  The following result was recently established by Bramson~\cite{BramsonSRBMinstabilityFluid}.
\begin{theorem}
There exists $\theta\in\R_+^6, \Sigma,R\in \R^{6\times 6}$ such that the fluid model $(\theta,R)$ is not globally stable,
but the underlying SRBM $(\theta,\Sigma,R)$ is positive recurrent.
\end{theorem}
On a positive side at least for the case $d\le 3$ the equivalence does take place. Moreover, in this case the stability
of SRBM and the global stability of the corresponding
 luid model can be constructively characterized in terms of $(\theta,R)$. The stability for the case $d=1$ is given simply
 by $\theta<0$ and was known for a while.
The case $d=2$ was resolved by El Kharroubi, Ben Tahar and Yaacoubi~\cite{Kharroubi2000}.
In a later paper El Kharroubi, Ben Tahar and Yaacoubi~\cite{Kharroubi2002} identified
exact conditions for global stability of a fluid model when $d=3$. The link with an SRBM when $d=3$
was resolved only recently in Bramson, Dai and Harrison~\cite{BramsonDaiHarrisonSRBM}.
\begin{theorem}
Suppose $d\le 3$. Then SRBM $(\theta,\Sigma,R)$ is positive recurrent if and only if the fluid model
$(\theta,R)$ is globally stable. Moreover, the stability of the fluid model $(\theta,R)$ can be verified
by checking a system of equalities and inequalities.
\end{theorem}
The second part of the theorem needs elaboration, but instead of explicitly giving the set of global stability conditions
we simply refer the reader to~\cite{BramsonDaiHarrisonSRBM}. For us the only relevant fact is that the stability condition
can be \emph{constructively} verified by checking a series of equations and inequalities in a rather straightforward way.
In light of this positive result, one would hope to extend this result for general $d$. We can summarize this question
as well as the state of the art as follows.

\begin{enumerate}
\item[\textbf{(a)}] When is an SRBM $(\theta,\Sigma,R)$ positive recurrent? The answer is known for $d=1,2,3$, but is unknown for general $d$.
\item[\textbf{(b)}] When is a fluid model $(\theta,R)$ globally stable? The answer is known for $d=1,2,3$ and coincides
with the answer for (a) for every non-singular $\Sigma$, but is unknown for general $d$.
\item[\textbf{(c)}] When is a fluid model $(z_0,\theta,R)$ stable?
\end{enumerate}

In this paper we resolve question (c), albeit in a somewhat unexpected way. We establish that this problem is undecidable and this is the
main result of this paper.

\begin{theorem}\label{theorem:mainresult}
The property "$(z_0,\theta,R)$ is stable" is algorithmically undecidable (non-computable). Namely, there does not exist an algorithm
which given an arbitrary input $z_0\in\R^d_+,\theta\in\R^d$ and a completely-$\mathcal{S}$ matrix $R$ outputs YES, if the fluid
model $(z_0,\theta,R)$ is stable and NO otherwise.
\end{theorem}

While we could only establish this result for problem (c), we conjecture that all problems (a)-(c) are undecidable.
The difficulty of establishing undecidability of the (b) problem is as follows.
While there is a well defined notion of a "global" halting property of a Counter Machine, unfortunately this does not appear to imply undecidability of
the global stability of $(\theta,R)$ property, since in our reduction there are starting states $z_0$ and the corresponding trajectory
which does not correspond to any trajectory of the Counter Machine. In other words the mapping $\Gamma\rightarrow (z_0,\theta,R)$
is not necessarily one-to-one as far as trajectories are concerned. Without analyzing the trajectories not corresponding to the
trajectories of the underlying Counter Machine it does not appear to be possible to use the reduction tool.

Proving undecidability of (a) also seems problematic since we need to simulate deterministic trajectories of a Counter Machine with
stochastic trajectories of an SRBM. In order to prove undecidability of the positive recurrence of SRBM,
it appears that one would need to construct an SRBM in which we can "control" stochastic trajectories for an
infinite amount of time. At the present moment we are not capable of doing this.

Theorem~\ref{theorem:mainresult} applies to Skorokhod problem in general dimension $d$. It would be natural to conjecture that the
stability of $(z_0,\theta,R)$ problem is undecidable for some fixed dimension $d$, for example when $d=6$ - the dimension
in which the counterexample of Theorem~\ref{theorem:counter-to-Skorokhod} takes place. In our reduction the dimension
$d$ of the Skorokhod problem depends directly on the number of states (see the next section) of the underlying Counter Machine.
It turns out that the Halting
problem of a Counter Machine is undecidable only when the number of states is general, since
there exist only finitely many Counter Machines with a given number of states. It is known
that any problem with a finite number of instances is decidable. This, however does not imply the same for the Skorokhod problem, since
the number of Skorokhod problems instances $(z_0,\theta,R)$ is infinite even for a \emph{fixed} dimension $d$, since there is an infinite
set of possibilities for the parameters $z_0,\theta$ and $R$ even for a fixed dimension $d$. All this means is that
our reduction technique does not lead to the instability result for the case of a fixed dimension $d$. We conjecture that there exists
a large enough, but dimension $d$ (perhaps $d=6$) in which the stability of a fluid model $(z_0,\theta,R)$ is undecidable.

\section{Counter Machine, Halting Problem and undecidability}\label{section:CounterAut}
A Counter Machine which we define below is a deterministic computing
machine which is a simplified version of a Turing Machine -- a
formal description of an algorithm performing a certain computational task or solving a certain decision problem.
(For a definition of a
Turing Machine see~\cite{sipser}).
We now define a Counter Machine and the Counter Machine Halting Problem (CMHP) which is known to be undecidable.
Our main technical result is the reduction of CMHP into the stability of a Skorokhod problem $(z_0,\theta,R)$
such that the Counter Machine halts if and only if $(z_0,\theta,R)$ halts.

A Counter Machine is described by 2 counters $R_1,R_2$ and a finite
collection of states $S=\{1,2,\ldots,m\}$. Each counter $R_i$ contains some nonnegative
integer  in its register. Depending on the current state $i\in S$ and
depending on whether the content of the registers is positive or
zero, the Counter Machine is updated as follows: the current state
$i$ is updated to a new state $j\in S$ and one of the counters has
its number in the register incremented by one, decremented by one or
no change in the counters occurs. More specifically, a  Counter Machine is a pair $(S=\{1,\ldots,m\},\Gamma)$.
where $\Gamma$ is
configuration update function $\Gamma:S\times \{0,1\}^2\rightarrow
S\times\{(-1,0),(0,-1),(0,0),(1,0),(0,1)\}$. A configuration of a Counter Machine is an
arbitrary triplet $(i,C_1,C_2)\in S\times \mathbb{Z}_+^2$. A
configuration $(i,C_1,C_2)$ is updated to a configuration
$(i,C_1',C_2')$ as follows.
Given a configuration $(i,C_1,C_2)$ suppose
$\Gamma(i,\mb{1}\{C_1>0\},\mb{1}\{C_2>0\})=(i',1,0)$.
Then the current state is changed from $i$
to $i'$, the content of the first counter   is incremented by one
and the second counter does not change: $C_1'=C_1+1, C_2'=C_2$. We
will also write $\Gamma:(i,C_1,C_2)\rightarrow (i',C_1+1,C_2)$.
Suppose, on the other hand,  $\Gamma(i,\mb{1}\{C_1>0\},\mb{1}\{C_2>0\})=(i',(-1,0))$. Then
the current state becomes $i'$, $C_1'=C_1-1,C_2'=C_2$. Similarly, if
$\Gamma(i,\mb{1}\{C_1>0\},\mb{1}\{C_2>0\})=(i',(0,1))$ or  $\Gamma(i,\mb{1}\{C_1>0\},\mb{1}\{C_2>0\})=(i',(0,-1))$, the new
configuration becomes $(i,C_1,C_2+1)$ or $(i',C_1,C_2-1)$,
respectively. If $\Gamma(i,\mb{1}\{C_1>0\},\mb{1}\{C_2>0\})=(i',(0,0))$ then the state is updated to
$i'$, but the contents of the counters do not change. It is assumed that the configuration update
function $\Gamma$ is consistent in the sense that it never attempts to decrement a counter which is equal to zero.
The present definition of a Counter Machine can be extended to the one which incorporates more than
two counters, but  such an extension is not necessary
for our purposes.

Given an initial configuration $(i^0,z^0_1,z^0_2)\in S\times \mathbb{Z}_+^2$ the Counter
Machine uniquely determines the subsequent configurations
$(i^1,z^1_1,z^1_2), (i^2,z^2_1,z^2_2),
\ldots,(i^t,z^t_1,z^t_2),\ldots~.$ We fix a certain
configuration $(i^*,C_1^*,C_2^*)\in S\times \Z_+^2$ and call it the \emph{halting}
configuration. If this configuration is reached then the process
halts and no additional updates are executed. The following theorem
establishes the undecidability (also called non-computability) of the halting property.
It is a classical result and can be found in~\cite{hooper}.

\begin{theorem}\label{theorem:counter_undecidable}
Given a Counter Machine $(S,\Gamma)$, initial configuration
$(i^0,C_1^0,C_2^0)$ and the halting configuration
$(i^*,C_1^*,C_2^*)$, the problem of determining whether the halting
configuration is reached in finite time (the Halting Problem) is undecidable.
Without the loss of generality it may be assumed that $C_1^*=C_2^*=0$.
\end{theorem}

We will need  to following small modification of the theorem above.

\begin{theorem}\label{theorem:counter_undecidable2}
Given a Counter Machine $(S,\Gamma)$, initial configuration
$(i^0,C_1^0,C_2^0)$ and a state $i^*$, the problem of determining whether the
state $i^*$ is reached in finite time  is undecidable. Moreover, without the loss
of generality it may be assumed that if and when the state $i^*$ is reached, we also have
$C_1=C_2=1$.
\end{theorem}

Namely this theorem states that determining whether a given \emph{state} (vs configuration) is reached
is undecidable as well, and if it is reached the counters take particular values (unity).
We call $i^*$ the halting state.

\begin{proof}
The proof is a simple reduction from the CMHT. Given a Counter Machine with the halting configuration $(i^*,0,0)$, augment
the set of states by a new state $m+1$. Define $\Gamma(i^*,0,0)=(m+1,1,1)$ (regardless of what it was before) and leave
$\Gamma(i^*,C_1,C_2)$ intact for all other $C_1,C_2$. Note that we have updated only the halting configuration.
Define $\Gamma(m+1,\cdot,\cdot)$ arbitrarily. Notice that state $m+1$ is reached in the augmented Counter Machine
if and only if the original Counter Machine halts starting from $(i^0,C_1^0,C_2^0)$. Notice also that by our augmented
rule if $m+1$ is reached at time $t$, then at this time both counters have values equal to unity.
\end{proof}

\section{A reduction of a Counter Machine into a Skorokhod problem}\label{section:Reduction}
Before we present the formal proof of our main result, Theorem~\ref{theorem:mainresult}, we describe a high level idea behind the proof.
In this and the following Section~\ref{section:dynamics} we consider an arbitrary Counter Machine
with $m$ states and construct a Skorokhod problem with the following properties.
The dimension of the Skorokhod problem is $d=5m+9$. Two of the coordinates of the $z(t)$ vector in the Skorokhod problem will
be "responsible" for the value of the counters in the Counter Machine, and $m$ other coordinates of $z(t)$ will be "responsible"
for the state of the Counter Machine. The remaining coordinates serve auxiliary roles. Specifically, the construction will be such
that if the counter machine has counter values $C_1$ and $C_2$ and is in state $i$ at time $t$, then coordinates of $z$ corresponding to counters
have also values  $C_1$ and $C_2$ at time $5t$, and the $m$ coordinates corresponding to the state have all values $1$ except for the coordinate
corresponding to the state $i$, which will have value zero at time $5t$. The time interval $[5t,5t+5]$  will correspond to emulating
the one step $t\rightarrow t+1$ of the Counter Machine. Thus we can "read off" the state of the Counter Machine from the state $z(t)$ of the
Skorokhod problem. This by itself does not prove Theorem~\ref{theorem:mainresult} since, as it turns out the Skorokhod problem
constructed this way is never stable, in the sense that it never enters the zero state. The completion of the proof is presented in
Section~\ref{section:wrapup} where we construct a small modification of our Skorokhod problem with the following property: if the
Counter Machine enters the halting state $i^*$ at some time $T$, then our modified Skorokhod problem enters zero state at time
$5T+1$. Alternatively, if the Counter Machine never enters state $i^*$, then our Skrokhod problem never enters the zero state.
Additionally, we show that the $R$ matrix corresponding to our modified Skorokhod problem is completely-$\mathcal{S}$. From this
construction we conclude that if we had an algorithm for checking whether a Skorokhod problem $(z_0,\theta,R)$ with a completely-$\mathcal{S}$
matrix $R$ is stable, we would also have an algorithm for checking whether a Counter Machine enters its halting state $i^*$, and thus
obtain a contradiction.

We now present details of the construction and the proof of Theorem~\ref{theorem:mainresult}.
Consider a Counter Machine with states $1,\ldots,m$ and starting configuration $(i^0,C_1^0,C_2^0)$.
We now construct a  Skorokhod problem (\ref{eq:z=x+Ry}),(\ref{eq:yincreasing}),(\ref{eq:non-idling}).
$x,y,z$ have dimension $d=5m+9$. $x(t)=z_0-\theta t,\theta\in\Z_+^d$
where $\theta_i=-1, 1\le i\le 5$ and $\theta_i=0$ otherwise, and $z_0$ will be specified later. Thus
\begin{align}\label{eq:beta}
\theta=\left(
        \begin{array}{c}
          -1 \\
          -1 \\
          -1 \\
          -1 \\
          -1 \\
          0 \\
          \vdots \\
          0 \\
        \end{array}
      \right)
\end{align}
We now describe matrix $R\in\Z^{d\times d}$.  The rows and columns of $R$ are grouped according to different roles they will be playing
in the construction. For convenience the groups are denoted by $A,B,C,D,E,F$, where the sizes are $5,m,2,2,4m$ and $4m$ respectively.
Specifically the first five rows and columns belong to the group $A$,  the next $m$ rows and columns belong to the group $B$, the following
two rows and columns are group $C$, etc.
The submatrix consisting of rows in one group and columns in some possibly other group will be denoted by a concatenation of the corresponding
letters. So, for example the $5\times 5$ submatrix $AA$ consists of rows and columns in $A$, $m\times 4m$
submatrix $BE$ consists of rows in $B$ and columns in $E$, etc.
Several  submatrices will be equal to zero and will be denoted by $0$. Several other square submatrices are identity matrices and are denoted by $I$.
In particular, our matrix $R$ is assumed to have the following structure:
\begin{align}\label{eq:R}
R=\left(
        \begin{array}{cccccc}
          AA & 0 & 0 & 0 & 0 & 0 \\
          BA & I & 0 & 0 & BE & 0 \\
          CA & 0 & I & I & CE & 0 \\
          DA & 0 & I & I & 0 & 0 \\
          EA & EB & EC & 0 & I& I \\
          FA & FB & FC & 0 & I& I \\
        \end{array}
      \right)
\end{align}
Let us now describe the remaining submatrices of the matrix $R$. $AA$ is set to be
\begin{align}\label{eq:AA}
AA=\left(
        \begin{array}{ccccc}
          1 & 2 & 1 & 1 & 0  \\
          0 & 1 & 2 & 1 & 1  \\
          1 & 0 & 1 & 2 & 1  \\
          1 & 1 & 0 & 1 & 2  \\
          2 & 1 & 1 & 0 & 1  \\
        \end{array}
      \right)
\end{align}
Namely $AA$ is a row vector $[1,2,1,1,0]$ rotated five times. From now on we adopt a notation of the form $XY_{i,j}$ where
$X,Y$ are one of the matrices $A,B,\ldots,F$ and $i$ and $j$ range over indices of $X$ and $Y$ respectively. So for example
$AA_{2,3}$ is the entry $(2,3)$ in the submatrix $AA$. Namely, $AA_{2,3}=2$.
We will also use notation of the form $XY_{i,*}$ to indicate a row vector corresponding to index $i$ in the submatrix $XY$.
Similarly, we will us $XY_{*,j}$, or $R_{i,*},R_{*,j}$ when the underlying matrix is entire $R$.

Let us describe the $m\times 5$ matrix $BA$. We set
$BA_{i,1}=-1$  and $BA_{i,j}=0$ for $2\le j\le 5$ for all $1\le i\le m$. As for the matrix $BE$,
for every $1\le i,j\le m$ and every $b,c\in \{0,1\}$, define $BE_{j,4i+2b+c}=0$ if $\Gamma(i,b,c)=(j,\Delta_1,\Delta_2)$ for some
updates $\Delta_1,\Delta_2\in \{-1,0,1\}$ and $=1$ otherwise, where $\Gamma$ is the configuration update function of the underlying
Counter Machine. In the matrix form we have
\begin{align}\label{eq:BABE}
BA=\left(
        \begin{array}{ccccc}
          -1 & 0 & 0 & 0 &  0  \\
          -1 & 0 & 0 & 0 &  0   \\
          \vdots & \vdots & \vdots & \vdots & \vdots \\
          -1 & 0 & 0 & 0 &  0   \\
        \end{array}
      \right),
\qquad
BE=\left(
        \begin{array}{ccccc}
          BE_{1,1} & BE_{1,2} &  \hdots &  BE_{1,4m}  \\
          BE_{2,1} & BE_{2,2} &  \hdots &  BE_{2,4m}  \\
          \vdots & \vdots & \ddots & \vdots \\
          BE_{m,1} & BE_{m,2} &  \hdots &  BE_{m,4m}  \\
        \end{array}
      \right)
\end{align}
Matrices $CA$ and $CE$ are $2\times 5$ and $2\times 4m$ respectively. For $i=1,2$ define
$CA_{i,1}=CA_{i,4}=-1$ and $CA_{i,j}=0$ otherwise. Also define
$CE_{1,4i+2b+c}=\Delta_1+1,CE_{2,4i+2b+c}=\Delta_2+1$ if $\Gamma(i,b,c)=(j,\Delta_1,\Delta_2)$.
Observe that $CE_{1,4i+2b+c},CE_{2,4i+2b+c}\ge 0$. Thus
\begin{align}\label{eq:CACE}
CA=\left(
        \begin{array}{ccccc}
          -1 & 0 & 0 & -1 &  0  \\
          -1 & 0 & 0 & -1 &  0   \\
        \end{array}
      \right),
\qquad
CE=\left(
        \begin{array}{ccccc}
          CE_{1,1} & CE_{1,2} &  \hdots &  CE_{1,4m}  \\
          CE_{2,1} & CE_{2,2} &  \hdots &  CE_{2,4m}  \\
        \end{array}
      \right)
\end{align}
Matrix $DA$ is  $2\times 5$. For $i=1,2$ define
$DA_{i,2}=-1$ and $DA_{i,j}=0$ otherwise. Thus:
\begin{align}\label{eq:DACE}
DA=\left(
        \begin{array}{ccccc}
          0 & -1 & 0 & 0 &  0  \\
         0 & -1 & 0 & 0 &  0  \\
        \end{array}
      \right)
\end{align}
Matrices $EA,EB$ and $EC$ are $4m\times 5, 4m\times m$ and $4m\times 2$ respectively. Define
$EA_{4i+2b+c,1}=-b-c, EA_{4i+2b+c,3}=-1$ and  $EA_{4i+2b+c,j}=0$ for $j=2,4,5$.
Define $EB_{4i+2b+c,i}=-1$, $EB_{4i+2b+c,j}=0$ for $j\ne i, ~1\le i,j\le m$. Finally,
define $EC_{4i+2b+c,1}=2b-1, EC_{4i+2b+c,2}=2c-1$. Namely
\begin{align}\label{eq:EAEC}
EA=\left(
        \begin{array}{ccccc}
          EA_{1,1} & 0 & -1 & 0 &  0  \\
          \vdots & \vdots & \vdots & \vdots &  \vdots   \\
          EA_{4m,1} & 0 & -1 & 0 &  0  \\
        \end{array}
      \right),
~~
EB=\left(
     \begin{array}{cccc}
       -1 & 0 & \hdots & 0\\
       -1 & 0 & \hdots & 0\\
       -1 & 0 & \hdots & 0\\
       -1 & 0 & \hdots & 0\\
       \vdots & \vdots & \ddots & \vdots\\
       0 & 0 & \hdots & -1 \\
       0 & 0 & \hdots & -1 \\
       0 & 0 & \hdots & -1 \\
       0 & 0 & \hdots & -1 \\
     \end{array}
   \right),
~~
EC=\left(
        \begin{array}{ccccc}
           EC_{1,1} & EC_{1,2} \\
          \vdots & \vdots \\
          EC_{4m,1} & EC_{4m,2} \\
        \end{array}
      \right)
\end{align}
Finally, matrices
$FA,FB$ and $FC$ are $4m\times 5, 4m\times m$ and $4m\times 2$ respectively. We set $FB=EB, FC=EC$.
Define
$FA_{4i+2b+c,1}=-b-c, FA_{4i+2b+c,3}=-1$,  $FA_{4i+2b+c,4}=-4$ and $FA_{4i+2b+c,5}=4$:
\begin{align}\label{eq:FAFC}
FA=\left(
        \begin{array}{ccccc}
          FA_{1,1} & 0 & -1 & -4 &  4  \\
          \vdots & \vdots & \vdots & \vdots &  \vdots   \\
          FA_{4m,1} & 0 & -1 & -4 &  4  \\
        \end{array}
      \right),
\qquad
FC=\left(
        \begin{array}{ccccc}
           FC_{1,1} & FC_{1,2} \\
          \vdots & \vdots \\
          FC_{4m,1} & FC_{4m,2} \\
        \end{array}
      \right)
\end{align}
This completes the description of the $R$ matrix. A natural question is whether $R$ is completely-$\mathcal{S}$. We defer this question to
Section~\ref{section:wrapup} where we first modify $R$ slightly and then show that indeed its modification is completely-$\mathcal{S}$.

\section{Dynamics of the Skorokhod problem}\label{section:dynamics}
Consistent with our notations $A,B,\ldots,F$, the coordinates of $\theta$ and the processes $x(t),y(t),z(t)$ are grouped by $A,B,\ldots,F$
as well. Again we will use the notation of the form $l_X, X=A,B,\ldots,F$ to denote a portion of a vector $l$ corresponding to the group $X$.
Let $l_{X,i}$ denote the $i$-th entry corresponding to the portion $X$. We will also sometimes use $l_{X,*}$ in place of $l_X$.
Specifically, as per construction in the previous section $\theta_A=\theta_{A,*}=(-1,\ldots,-1)^T,\theta_B=0,\ldots,\theta_F=0$.
Equivalently $\theta_{A,1}=\ldots=\theta_{A,5}=-1,\theta_{B,1}=\ldots=\theta_{B,m}=0$, etc. Also $\theta_X=0$ for $X=B,C,D,E,F$.

Now let us define the initial state $x(0)=z(0)\equiv z_0$ of the process $x(t)$.
Suppose  the configuration of the Counter Machine at time$t=0$ is $(i^0,C^0_1,C^0_2)$.
Then we set $x_{A,1}(0)=x_{A,5}(0)=0,x_{A,2}(0)=x_{A,3}(0)=x_{A,4}(0)=1$,
$x_{B,i^0}(0)=0$ and  $x_{B,j}(0)=1$ for all $j\ne i^0, 1\le j\le m$. We set $x_{C,1}(0)=C^0_1,x_{C,2}(0)=C^0_2$. $x_{D,1}(0)=x_{D,2}(0)=0$.
Finally $x_{E,4i+2b+c}(0)=3, x_{F,4i+2b+c}(0)=4$ for all $1\le i\le m, b,c,\in \{0,1\}$. Namely, in the vector form
\begin{align}\label{eq:x0}
&x_A(0)=\left(
        \begin{array}{c}
          0 \\
          1 \\
          1 \\
          1 \\
          0 \\
        \end{array}
      \right),
~~
x_B(0)=\left(
        \begin{array}{c}
          1 \\
          \vdots \\
          1 \\
          0 \\
          1 \\
          1 \\
          \vdots \\
          1 \\
        \end{array}
      \right),
~~
x_C(0)=\left(
        \begin{array}{c}
          C^0_1 \\
          C^0_2 \\
        \end{array}
      \right),
      ~~
& x_D(0)=\left(
        \begin{array}{c}
          0 \\
          0 \\
        \end{array}
      \right),
      ~~
x_E(0)=\left(
        \begin{array}{c}
          3 \\
          3 \\
          \vdots \\
          3 \\
        \end{array}
      \right),
\\
&x_F(0)=\left(
        \begin{array}{c}
          4 \\
          4 \\
          \vdots \\
          4 \\
        \end{array}
      \right) \notag
\end{align}

The next theorem is our key technical result. It shows that the configuration of a Counter Machine at times $t=0,1,\ldots~$
can be encoded by a Skorokhod problem just constructed at times $5t, ~t=0,1,2,\ldots~.$
\begin{theorem}\label{theorem:counter-to-Skorokhod}
Given a Counter Machine $(S,\Gamma)$ with a starting configuration $(i^0,C^0_1,C^0_2)$
suppose its configuration at time $t$ is $(i^t,C^t_1,C^t_2), ~t=0,1,\ldots~$.
Then at time $5t$ the state $z(5t)$ of the corresponding Skorokhod problem is as follows:
$z_{A,1}(5t)=z_{A,5}(5t)=0,z_{A,2}(5t)=z_{A,3}(5t)=z_{A,4}(5t)=1$,
$z_{B,i^t}(5t)=0$ and $z_{B,i}(5t)=1$ for all $i\ne i^t$, $z_{C,k}(5t)=C^t_k, ~k=1,2$,
$z_{D,*}(5t)=0,z_{E,*}(5t)=3,z_{F,*}(5t)=4$. In particular, $z_{B,*}$ and $z_{C,*}$ encode the state $i^t$ and the counters $C^t_1,C^t_2$
of the Counter Machine at time $5t$, respectively.
\end{theorem}
In vector form the claimed state $z(5t)$ is thus described in (\ref{eq:x0}), where $x$ replaces $z$, the unique $0$
in the $B$-component of $z$ corresponds to index $i^t$, and $C^t_1,C^t_2$ replace $C^0_1,C^0_2$.

By itself this theorem does not prove our main result Theorem~\ref{theorem:mainresult}: it is easy to see that
$z(t)\neq 0$ for all $t$. Namely, our Skorokhod problem is not stable. In Section~\ref{section:wrapup} we construct
a modification of $R$ such that the halting state $i^*$ is reached if and only if the Skorokhod problem is stable.
Theorem~\ref{theorem:mainresult} then will follow from Theorem~\ref{theorem:counter_undecidable2}.

\begin{proof}
The remainder of this section is devoted to the proof of Theorem~\ref{theorem:counter-to-Skorokhod}.
The argument is based on induction. The base case is covered by our assumptions on the initial state $x(0),y(0),z(0)$.
For the induction part we simply carefully track
the dynamics of the Skorokhod problem over a time interval $[5t,5t+5]$ and show that if at time $5t$ the Skorokhod
problem encodes the configuration of the Counter Machine at time $t$, then it will also do so at the time $5t+5$
for the configuration at  time $t+1$.

We now provide details of the induction argument. We begin by separately analyzing the dynamics of $z_A(t)$. Then
we analyze the dynamics of the remaining part of the process $z(t)$ over intervals $[5t,5t+1],\ldots,[5t+4,5t+5]$ separately.

\subsection{Dynamics of $z_A(t)$}
As we now show the dynamics of $z_{A,*},y_{A,*}$ is periodic, does not depend on the dynamics of the Counter Machine
and does not depend on the dynamics of the rest of the vectors $z,y$, in some appropriate sense.
\begin{lemma}\label{lemma:xA dynamics}
For every $t\in\Z_+$ and $i\in \{1,2,3,4\}$, $z_{A,i}(5t+i)=z_{A,i+1}(5t+i)=0$ and $z_{A,j}(5t+i)=1$ for all $j\ne i,i+1$.
Also $z_{A,1}(5t)=z_{A,5}(5t)=0$, and $z_{A,j}(5t)=1$ for all $j\ne 1,5$. Namely, vector $z_A(t)$ cycles through the
following five vectors over the time instances $5t,5t+1,\ldots,5t+4$:
\begin{align}\label{eq:x01}
\left(\begin{array}{c}
          0 \\
          1 \\
          1 \\
          1 \\
          0 \\
        \end{array}
      \right)
~~
\left(\begin{array}{c}
          0 \\
          0 \\
          1 \\
          1 \\
          1 \\
        \end{array}
      \right)
~~
\left(\begin{array}{c}
          1 \\
          0 \\
          0 \\
          1 \\
          1 \\
        \end{array}
      \right)
~~
\left(\begin{array}{c}
          1 \\
          1 \\
          0 \\
          0 \\
          1 \\
        \end{array}
      \right)
~~
\left(\begin{array}{c}
          1 \\
          1 \\
          1 \\
          0 \\
          0 \\
        \end{array}
      \right)
\end{align}
Moreover, for every $t\in\Z_+$ and $i=0,1,2,3,4$, variable $y_{A,i+1}$ is active at the unit rate
in $(5t+i,5t+i+1)$, and is passive otherwise.  In particular, during the time period
$(5t+i,5t+i+1), ~0\le i\le 4$, the variable $y_{A,i+1}$ is the unique active variable among five variables $y_{A,*}$.
\end{lemma}

\begin{proof}
Observe that no  variables other than $y_{A,*}$ influence $z_{A,*}$. Thus it suffices to establish the result
for $t=0$ and demonstrate the periodicity property $z_A(5)=z_A(0)$. First consider time period $[0,1]$.
Recall that $z_{A,1}=z_{A,5}=0$ and $z_{A,i}=1$ for $i\ne 1,5$. From (\ref{eq:z=x+Ry}) we have for every $s\in (0,1)$
\begin{align*}
z_{A,2}(s)&=z_{A,2}(0)-\theta_{A,2}s+y_{A,2}(s)-y_{A,2}(0)+2(y_{A,3}(s)-y_{A,3}(0))+y_{A,4}(s)-y_{A,4}(0)\\
&+y_{A,5}(s)-y_{A,5}(0) \\
&\ge 1-s \\
&>0,
\end{align*}
where the non-decreasing property (\ref{eq:yincreasing}) of $y_j$ was used. The constraint (\ref{eq:non-idling}) then
implies that $y_{A,2}$ remains passive in the interval $(0,1)$. Similarly we show that variables $y_{A,3}$ and $y_{A,4}$
remain passive over the same time interval.
Applying again (\ref{eq:z=x+Ry}) we obtain for every $s\in (0,1)$
\begin{align*}
z_{A,1}(s)&=0-s+y_{A,1}(s)-y_{A,1}(0), \\
z_{A,5}(s)&=0-s+2(y_{A,1}(s)-y_{A,1}(0))+y_{A,5}(s)-y_{A,5}(0).
\end{align*}
By (\ref{eq:non-idling}) and non-negativity of $z$,
the first identity implies that
$y_{A,1}(s)=s$. Applying this to the second identity we obtain that $z_{A,5}$ becomes strictly positive for $s>0$ and therefore
the variable $y_{A,5}$ is passive over the interval $(0,1)$.
We conclude that for $s\in (0,1)$, $z_{A,2}(s)=1-s,  z_{A,3}(s)=1-s+s=1, z_{A,4}(s)=1-s+s=1,z_{A,5}(s)=0-s+2s$. In particular,
$z_{A,1}(1)=z_{A,2}(1)=0, z_{A,3}(1)=z_{A,4}(1)=z_{A,5}(1)=1$. This establishes the claim for $s\in (0,1)$. The proof for the cases
$s\in (1,2),(2,3),(3,4),(4,5)$ is obtained by shifting the indices by one.
\end{proof}

\subsection{Interval $[5t,5t+1]$}
In this and the remaining  subsections we will say that a variable $y_j$ is active when it is active at the unit rate.
When a variable is active at a non-unit rate we will explicitly say this.
For every $i=1,2,\ldots,m$ and $b,c=0,1$ let
\begin{align}\label{eq:Upsilon}
\Upsilon(i,b,c)=-b-c+(2b-1)\delta(C^t_1=0)+(2c-1)\delta(C^t_2=0)-\delta(i=i^t).
\end{align}

\begin{lemma}\label{lemma:5t 5t+1 positive counters}
In the time period $(5t,5t+1)$ variables $y_{A,1},y_{B,i^t}$ are active.
Variable $y_{C,i}, i=1,2$ is active
iff $C^t_i=0$. All the remaining $y$ variables are passive.
The state $z(5t+1)$ is as follows:  $z_{B,*}(5t+1)=0$,
$z_{C,i}(5t+1)=C^t_i-1+\delta(C^t_i=0), ~i=1,2$, $z_{D,i}(5t+1)=\delta(C^t_i=0),~ i=1,2$,
$z_{E,4i+2b+c}(5t+1)=3+\Upsilon(i,b,c)$ for all $i=1,\ldots,m$, and
$z_{F,4i+2b+c}(5t+1)=4+\Upsilon(i,b,c)$ for all $i=1,\ldots,m$ and $b,c=0,1$.
Namely, in the vector form

\begin{align*}
& z_B(5t+1)=\left(
        \begin{array}{c}
          0 \\
          0 \\
          \vdots \\
          0 \\
        \end{array}
      \right)
\qquad z_C(5t+1)=\left(
        \begin{array}{c}
          C^t_1-1+\delta(C^t_1=0) \\
          C^t_2-1+\delta(C^t_2=0) \\
        \end{array}
      \right)
\qquad z_D(5t+1)=\left(
        \begin{array}{c}
          \delta(C^t_1=0) \\
          \delta(C^t_2=0) \\
        \end{array}
      \right)
\\
& z_E(5t+1)=\left(
        \begin{array}{c}
          \vdots \\
          3+\Upsilon(i,b,c) \\
          \vdots \\
        \end{array}
      \right)\\
& z_F(5t+1)=\left(
        \begin{array}{c}
          \begin{array}{c}
          \vdots \\
          4+\Upsilon(i,b,c)\\
          \vdots \\
        \end{array}
        \end{array}
      \right)
      &
\end{align*}
\end{lemma}

\begin{proof}
We begin by identifying the active variables. Among $y_{A,i}$ we already established in Lemma~\ref{lemma:xA dynamics} that only $y_{A,1}$ is active.
Consider now variables $y_{B,i}$ for $i\ne i^t$. Since by our inductive assumption
$z_{B,i}(5t)=1$, since $BA_{i,1}=-1$ and all the remaining entries of rows corresponding to $B$
are non-negative, then $z_{B,i}$ remain positive in $(5t,5t+1)$ and therefore, all $y_{B,i}$ remain passive during this time interval.
As for the variable $z_{B,i^t}$ which equals zero at time $t=0$, we have for $s\in (0,1)$
\begin{align*}
z_{B,i^t}(5t+s)&=-(y_{A,1}(5t+s)-y_{A,1}(5t))+(y_{B,i^t}(5t+s)-y_{B,i^t}(5t))\\
&+\sum_{1\le j\le 4m}BE_{i^t,j}(y_{E,j}(5t+s)-y_{E,j}(5t)).
\end{align*}
Since  $BE\ge 0$, and the rate of $y_{A,1}$ is unity, we see $y_{B,i^t}(s)$ increases at the rate at most unity.
Later on we will show that in fact $y_{E,*}$ remain passive and therefore $y_{B,i^t}(s)$ increases precisely at the unit rate.

Consider variables $y_{C,i}, i=1,2$. We have for $s\in (0,1)$ by our inductive assumption
\begin{align}
z_{C,i}(5t+s)&=C^t_i-(y_{A,1}(5t+s)-y_{A,1}(5t))+(y_{C,i}(5t+s)-y_{C,i}(5t)) \label{eq:zC5t}\\
&+(y_{D,i}(5t+s)-y_{D,i}(5t))\\
&+\sum_{1\le j\le 4m}CE_{i,j}(y_{E,j}(5t+s)-y_{E,j}(5t)). \notag
\end{align}
Recall that the entries of $CE$ are non-negative. The rate of $y_{A,1}$ is unity. If $C^t_i\ge 1$, this implies
that variables $z_{C,i}$ remains positive on $[5t,5t+1)$ and therefore $y_{C,i}$ is passive. If $C^t_i=0$, then, as in the case
of $y_{B,i^t}$ variable, $y_{C,i}$ can be active at the rate at most unity. When we show that $y_{D,*}$ and $y_{E,*}$ are passive, this will
imply that $y_{C,i}$ is active at the unit rate.

Consider variables $y_{D,i}$. The only negative entries in rows corresponding to $D$ are $DA_{i,2}$. However, $y_{A,2}$ is passive. Thus
$y_{D,i}, i=1,2$ remain passive.

Consider variables $y_{E,*}$. We have established that $y_{B,i}$ are passive for $i\neq i^t$,  $y_{C,i}$ are passive unless $C^t_i=0$ and $y_{D,*}$ are passive.
The entries of $EE$ and $EF$ are non-negative. Therefore
\begin{align}
z_{E,4i+2b+c}(5t+s)&\ge 3-(b+c)(y_{A,1}(5t+s)-y_{A,1}(5t)) \label{eq:zE5t}\\
&-\delta(i=i^t)(y_{B,i^t}(5t+s)-y_{B,i^t}(5t)) \notag\\
&+\delta(C^t_1=0)(2b-1)(y_{C,1}(5t+s)-y_{C,1}(5t))\notag\\
&+\delta(C^t_2=0)(2c-1)(y_{C,2}(5t+s)-y_{C,2}(5t)). \notag
\end{align}
Let us verify case by case that $z_{E,4i+2b+c}$ remains positive in $[5t,5t+1)$. We know that $y_{A,1}$ increases at the unit rate.
We have already established that $y_{B,i^t},y_{C,i}$ increase at the rate at most unity. If $b=c=0$, giving $2b-1=2c-1=-1$ we  obtain
$z_{E,4i+2b+c}(5t+s)\ge 3-(1+1+1)s>0$. If $b=1,c=0$, we again obtain $z_{E,4i+2b+c}(5t+s)\ge 3-(1+1+1)s>0$. Similarly when $b=0,c=1$.
Finally $b=c=1$, we also obtain $z_{E,4i+2b+c}(5t+s)\ge 3-3s>0$. We have established  that $z_{E,4i+2b+c}$ remains positive.
This implies that $y_{E,*}$ remain passive. This further implies that $y_{B,i^t}$ is active at unit rate and $y_{C,i}$ is
also active at unit rate if $C^t_i=0$.

The case of $y_{F,*}$ variables is similar to the one of $y_{E,*}$ in light of the fact that at time zero these variables have value $4$
as opposed variables $y_{E,*}$ which have value $3$.
We conclude that the only active variables are $y_{A,1}$, $y_{B,*}$ and $y_{C,i}$ when $C^t_i=0$, and they increase at a unit rate.

It is now easy to identify the state $z(5t+1)$. For $s\in (0,1)$ we have $z_{B,i}(5t+s)=1-(y_{A,1}(5t+s)-y_{A,1}(5t))=1-s$ for $i\ne i^t$;
$z_{B,i^t}(5t+s)=0-(y_{A,1}(5t+s)-y_{A,1}(5t))+(y_{B,i^t}(5t+s)-y_{B,i^t}(5t))=0$. Therefore $z_{B,*}(5t+1)=0$.
We may rewrite (\ref{eq:zC5t}) as
\begin{align*}
z_{C,i}(5t+s)=C^t_i-s+(y_{C,i}(5t+s)-y_{C,i}(5t)).
\end{align*}
We obtain $z_{C,i}=C^t_i-1$ when $C^t_i\ge 1$ and $=0$ otherwise. We can write this as $z_{C,i}=C^t_i-1+\delta(C^t_i=0)$.

For $z_{D,*}$ variables we have
\begin{align*}
z_{D,i}(5t+s)=y_{C,i}(5t+s)-y_{C,i}(5t).
\end{align*}
Thus $z_{D,i}(5t+1)=1$ if $y_{C,i}$ is active, namely if $C^t_i=0$ and $=0$ otherwise.
In short, $z_{D,i}(5t+1)=\delta(C^t_i=0)$.
Rewrite (\ref{eq:zE5t}) as
\begin{align*}
z_{E,4i+2b+c}(5t+s)&= 3-(b+c)s-\delta(i=i^t)s
+\delta(C^t_1=0)(2b-1)s+\delta(C^t_2=0)(2c-1)s.
\end{align*}
Evaluating at  $s=1$ we obtain $z_{E,*}(5t+1)=3+\Upsilon(i,b,c)$.

Finally, we have
\begin{align*}
z_{F,4i+2b+c}(5t+s)&= 4-(b+c)s-\delta(i=i^t)s
+\delta(C^t_1=0)(2b-1)s+\delta(C^t_2=0)(2c-1)s.
\end{align*}
Evaluating at $s=1$ we obtain  $z_{F,*}(5t+1)=4+\Upsilon(i,b,c)$.
\end{proof}

\subsection{Interval $[5t+1,5t+2]$}
We will show that the only change occurring in this time interval is increasing the values of $z_{C,1},z_{C,2}$ by one.
\begin{lemma}\label{lemma:5t+1 5t+2 positive counters}
In the time interval $(5t+1,5t+2)$ variable $y_{A,2}$ is active. $y_{D,i}$ is  active if and only if $C^t_i>0$ for $i=1,2$.
All the remaining $y$ variables are passive.
Moreover,
$z_{C,i}(5t+2)=C^t_i,~i=1,2$, and  $z_{X,*}(5t+2)=z_{X,*}(5t+1)$ for $X=A,B,D,E,F$.
\end{lemma}

\begin{proof}
The only negative entries in $R$ in rows corresponding to  $B$ are $BA_{1,*}$. By Lemma~\ref{lemma:xA dynamics} $y_{A,1}$ is passive.
Thus variables $y_{B,*}$ remain passive. The same applies to variables $y_{C,*}$.
We skip the analysis of $y_{D,*}$ for a moment. Observe that $y_{E,*},y_{F,*}$ remain passive
since by Lemma~\ref{lemma:xA dynamics} $y_{A,1}$ and $y_{A,3}$ are passive, and as we have established above, $y_{B,*}$ are passive.
It remains to analyze $y_{D,*}$. We obtain for $s\in (0,1)$
\begin{align*}
z_{D,i}(5t+1+s)&=z_{D,i}(5t+1)-(y_{A,2}(5t+1+s)-y_{A,2}(5t+1))+(y_{D,i}(5t+1+s)-y_{D,i}(5t+1))\\
&=\delta(C^t_i=0)-s+(y_{D,i}(5t+1+s)-y_{D,i}(5t+1)),
\end{align*}
which immediately implies that $y_{D,i}$ is active if $C^t_i>0$ and passive otherwise.

It is easy now to identify the state at time $5t+2$. By direct inspection we obtain that
for $s\in (0,1)$
\begin{align*}
z_{C,i}(5t+1+s)&=z_{C,i}(5t+1)+y_{D,i}(5t+1+s)-y_{D,i}(5t+1)\\
&=C^t_i-1+\delta(C^t_i=0)+s\delta(C^t_i>0) \\
&=C^t_i.
\end{align*}
All the remaining variables $z_{*}$ remain the same. This concludes the proof.
\end{proof}

\subsection{Interval $[5t+2,5t+3]$}

\begin{lemma}\label{lemma:5t+2 5t+3 positive counters}
In the time interval $(5t+2,5t+3)$ the variable $y_{A,3}$ is active and variable $y_{E,4i^t+2b+c}$ is active
for $b=\min(C^t_1,1),c=\min(C^t_2,1)$. All the remaining  variables are passive.
Moreover,  at time $5t+3$ the state $z(5t+3)$ is as follows:  $z_{B,i^{t+1}}(5t+3)=0$ and $z_{B,i}(5t+3)=1$ for $i\ne i^{t+1}$.
$z_{C,i}(5t+3)=C^t_i+\Delta_i+1, i=1,2$. $z_{D,i}(5t+3)=0,~ i=1,2$. Finally,
$z_{E,4i+2b+c}(5t+3)=3+\Upsilon(i,b,c)+\mb{1}\{(i,b,c)=(i^t,b^t,c^t)\}$ and
$z_{F,4i+2b+c}(5t+3)=4+\Upsilon(i,b,c)+\mb{1}\{(i,b,c)=(i^t,b^t,c^t)\}$.
Namely, in the vector form

\begin{align*}
& z_B(5t+3)=\left(
        \begin{array}{c}
          1 \\
          1 \\
          \vdots \\
          1\\
          0\\
          1\\
          \vdots\\
          1 \\
        \end{array}
      \right)
\qquad z_C(5t+3)=\left(
        \begin{array}{c}
          C^t_1+\Delta_1+1 \\
          C^t_2+\Delta_2+1 \\
        \end{array}
      \right)
\qquad z_D(5t+3)=\left(
        \begin{array}{c}
          0 \\
          0 \\
        \end{array}
      \right)
\\
& z_E(5t+3)=\left(
        \begin{array}{c}
          \vdots \\
          2+\Upsilon(i,b,c))+\mb{1}\{(i,b,c)=(i^t,b^t,c^t)\} \\
          \vdots \\
        \end{array}
      \right)\\
& z_F(5t+3)=\left(
        \begin{array}{c}
          \begin{array}{c}
          \vdots \\
          3+\Upsilon(i,b,c)+\mb{1}\{(i,b,c)=(i^t,b^t,c^t)\}\\
          \vdots \\
        \end{array}
        \end{array}
      \right).
\end{align*}

\end{lemma}

\begin{proof}
By non-negativity of $BE$ and $CE$ sub-matrices and since $y_{A,3}$ is the only active among $A$ variables, we verify directly
that $y_{B,*},y_{C,*},y_{D,*}$ remain passive. Now consider variables $y_{E,4i+2b+c}$ with $i\ne i^t$. By
Lemmas~\ref{lemma:5t 5t+1 positive counters},\ref{lemma:5t+1 5t+2 positive counters},  we have
$z_{E,4i+2b+c}(5t+2)=3-b-c+(2b-1)\delta(C^t_1=0)+(2c-1)\delta(C^t_2=0)$ for $i\ne i^t$.
Observe that this quantity is always at least a unity.
Since variables $y_{B,*},y_{C,*}$ are passive, by inspecting rows $E$ of matrix $R$ we obtain that for all $s\in (0,1)$
\begin{align*}
z_{E,4i+2b+c}(5t+2+s)&\ge 1-(y_{A,3}(5t+2+s)-y_{A,3}(5t+2))\\
&\ge 1-s \\
&>0.
\end{align*}
Therefore variables $y_{E,4i+2b+c}$ remain passive. Now consider $i=i^t$. In this case we have
$z_{E,4i^t+2b+c}(5t+2)=2-b-c+(2b-1)\delta(C^t_1=0)+(2c-1)\delta(C^t_2=0)$.
Case by case inspection shows that $-b+(2b-1)\delta(C^t_1=0)=-1$ when $b=\min(C^t_1,1)$ and $\ge 0$ otherwise.
Similarly $-c+(2c-1)\delta(C^t_2=0)=-1$ when $c=\min(C^t_2,1)$. Therefore $y_{E,4i^t+2b+c}$ remains
passive unless $b=\min(C^t_1,1)$ and  $c=\min(C^t_2,1)$.
Before we turn to $y_{E,4i^t+2b+c}$ when $b=\min(C^t_1,1), c=\min(C^t_2,1)$,
we need to analyze $y_{F,*}$.

For these variables we recall from Lemmas~\ref{lemma:5t 5t+1 positive counters},\ref{lemma:5t+1 5t+2 positive counters}
that $z_{F,4i+2b+c}(5t+2)=4+\Upsilon(i,b,c)\ge 1$ for all $i$. Therefore, for $s\in (0,1)$
\begin{align*}
z_{F,4i+2b+c}(5t+2+s)\ge 1-(y_{A,3}(5t+2+s)-y_{A,3}(5t+2))=1-s>0.
\end{align*}
We conclude that $y_{F,*}$ remain passive.

Let us return to $y_{E,4i^t+2b+c}$ when $b=\min(C^t_1,1), c=\min(C^t_2,1)$. Observe that in this case $\Upsilon(i,b,c)=-3$. Therefore
\begin{align*}
z_{E,4i^t+2b+c}(5t+2+s)&= 3+\Upsilon(i,b,c)-s+y_{E,4i^t+2b+c}(5t+2+s)-y_{E,4i^t+2b+c}(5t+2)\\
&=-s+y_{E,4i^t+2b+c}(5t+2+s)-y_{E,4i^t+2b+c}(5t+2).
\end{align*}
Therefore $y_{E,4i^t+2b+c}$ is active in the interval $(5t+2,5t+3)$.

It is now easy to compute the state at time $5t+3$. Let $b^t=\min(C^t_1,1), c^t=\min(C^t_2,1)$.
We have from the definition of $BE$ sub-matrix and from the state of the system at time $5t+2$
that for all $s\in (0,1)$ and $i\ne i^{t+1}$
\begin{align*}
z_{B,i}(5t+2+s)=y_{E,4i^t+2b^t+c^t}(5t+2+s)-y_{E,4i^t+3}(5t+2)=s
\end{align*}
and $z_{B,i^{t+1}}(5t+2+s)=0$. Evaluating at $s=1$, we obtain the claim for $z_{B,*}(5t+3)$.
We have for $i=1,2$
\begin{align*}
z_{C,i}(5t+2+s)&=C^t_i+(\Delta_1+1)(y_{E,4i^t+2b^t+c^t}(5t+2+s)-y_{E,4i^t+2b^t+c^t}(5t+2))\\
&=C^t_i+(\Delta_i+1)s.
\end{align*}
Evaluating at $s=1$, we obtain $z_{C,i}(5t+3)=C^t_i+\Delta_i+1$.

Variables $z_{D,*}$ remain unchanged.

Now we analyze $z_{E,*}$. We have for $4i+2b+c\ne 4i^t+2b^t+c^t$
\begin{align*}
z_{E,4i+2b+c}(5t+2+s)&=3+\Upsilon(i,b,c)-s.
\end{align*}
and $z_{E,4i^t+2b^t+c^t}(5t+2+s)=0$. Recalling that $\Upsilon(i^t,b^t,c^t)=-3$, we may combine two cases as
\begin{align*}
z_{E,4i+2b+c}(5t+3)=2+\Upsilon(i,b,c)+\mb{1}\{(i,b,c)=(i^t,b^t,c^t)\}.
\end{align*}

Finally, for $4i+2b+c\ne 4i^t+2b^t+c^t$ we have
\begin{align*}
z_{F,4i+2b+c}(5t+2+s)=4+\Upsilon(i,b,c)-s,
\end{align*}
and
\begin{align*}
z_{F,4i^t+2b^t+c^t}(5t+2+s)&=4+\Upsilon(i^t,b^t,c^t)-s+s\\
&=1.
\end{align*}
Again we may combine the cases by saying
\begin{align*}
z_{F,4i+2b+c}(5t+3)=3+\Upsilon(i,b,c)+\mb{1}\{(i,b,c)=(i^t,b^t,c^t)\}.
\end{align*}
This completes the proof of the lemma.
\end{proof}

\subsection{Interval $[5t+3,5t+4]$}

\begin{lemma}\label{lemma:5t+3 5t+4 positive counters}
In the time period $(5t+3,5t+4)$ variable $y_{A,4}$ is active. The variables $y_{F,4i+2b+c}$ are passive in the time interval
$(5t+3,5t+3+(3+\Upsilon(i,b,c)+\mb{1}\{(i,b,c)=(i^t,b^t,c^t)\})/4)$ and active during the time interval
$(5t+3+(3+\Upsilon(i,b,c)+\mb{1}\{(i,b,c)=(i^t,b^t,c^t)\})/4,5t+4)$.
All the remaining variables are passive in $(5t+3,5t+4)$.
The state $z(5t+4)$ is as follows.  $z_{B,i^{t+1}}(5t+4)=0$ and $z_{B,i}(5t+4)=1$ for $i\ne i^{t+1}$,
$z_{C,i}(5t+4)=C^t_i+\Delta_i, i=1,2$, $z_{D,i}(5t+3)=0,~ i=1,2$,  $z_{E,*}(5t+4)=3$, and $z_{F,*}(5t+4)=0$.
Namely, in vector form

\begin{align*}
& z_B(5t+4)=\left(
        \begin{array}{c}
          1 \\
          1 \\
          \vdots \\
          1\\
          0\\
          1\\
          \vdots\\
          1 \\
        \end{array}
      \right)
\qquad z_C(5t+4)=\left(
        \begin{array}{c}
          C^t_1+\Delta_1 \\
          C^t_2+\Delta_1 \\
        \end{array}
      \right)
\qquad z_D(5t+4)=\left(
        \begin{array}{c}
          0 \\
          0 \\
        \end{array}
      \right)
\\
& z_E(5t+4)=\left(
        \begin{array}{c}
          \vdots \\
          3 \\
          \vdots \\
        \end{array}
      \right)
\qquad z_F(5t+4)=\left(
        \begin{array}{c}
          \begin{array}{c}
          \vdots \\
          0\\
          \vdots \\
        \end{array}
        \end{array}
      \right)
      &
\end{align*}

\end{lemma}

\begin{proof}
Variables $y_{B,*}$ are passive since $BA_{*,4}=0$ and the entries of $BE$ are non-negative. For $y_{C,*}$ variables we have
by non-negativity of $CE$ that for $s\in (0,1)$
\begin{align*}
z_{C,i}(5t+3+s)\ge z_{C,i}(5t+3)-(y_{A,4}(5t+3+s)-y_{A,4}(5t+3))=z_{C,i}(5t+3)-s>0,
\end{align*}
since by Lemma~\ref{lemma:5t+2 5t+3 positive counters} we have $z_{C,i}(5t+3)=C^t_i+\Delta_i+1\ge 1$
(recall that by feasibility of the Counter Machine we have $C^t_i+\Delta_i\ge 0$). Thus $y_{C,*}$ remain passive.
Variables $y_{D,*}$ remain passive since $DA_{i,4}=0$ and all other entries in rows of $R$ corresponding to $D$ are non-negative.
Variables $y_{E,*}$ remain passive since $EA_{*,4}=0$, variables $y_{B,*}$ and $y_{C,*}$ are passive and all other entries in rows of $R$
corresponding to $E$ are non-negative.

Finally, we analyze $y_{F,*}$ variables. Applying Lemma~\ref{lemma:5t+2 5t+3 positive counters}, and the fact that $y_{B,*},y_{C,*},y_{D,*}$ and
$y_{E,*}$ variables are passive,
we have for $s\in (0,1)$
\begin{align*}
z_{F,4i+2b+c}(5t+3+s)&= z_{F,4i+2b+c}(5t+3)-4(y_{A,4}(5t+3+s)-y_{A,4}(5t+3))\\
&+(y_{F,4i+2b+c}(5t+3+s)-y_{F,4i+2b+c}(5t+3))\\
&=3+\Upsilon(i,b,c)+\mb{1}\{(i,b,c)=(i^t,b^t,c^t)\}\\
&-4s+(y_{F,4i+2b+c}(5t+3+s)-y_{F,4i+2b+c}(5t+3))
\end{align*}
It is easy to verify directly that $\Upsilon(i,b,c)\le 0$ and in the case $(i,b,c)=(i^t,b^t,c^t)$, additionally $\Upsilon(i,b,c)=-3$.
Therefore $3+\Upsilon(i,b,c)+\mb{1}\{(i,b,c)=(i^t,b^t,c^t)\}\le  3$. It follows that
\begin{align*}
3+\Upsilon(i,b,c)+\mb{1}\{(i,b,c)=(i^t,b^t,c^t)\}-4s
\end{align*}
remains positive for $s<(3+\Upsilon(i,b,c)+\mb{1}\{(i,b,c)=(i^t,b^t,c^t)\})/4.$
Starting with  $s=(3+\Upsilon+\mb{1}\{(i,b,c)=(i^t,b^t,c^t)\})/4$,
variables $y_{F,4i+2b+c}$ will start increasing at the rate $4$, while $z_{F,4i+2b+c}(5t+3+s)$ remains zero. This completes the analysis of active/passive variables.

We now compute the state at time $5t+4$. Variables $z_{B,*}$ do not change over the time interval $[5t+3,5t+4]$.
We have for $s\in (0,1)$
\begin{align*}
z_{C,i}(5t+3+s)=z_{C,i}(5t+3)-(y_{A,4}(5t+3+s)-y_{A,4}(5t+3))=z_{C,i}(5t+3)-s.
\end{align*}
We obtain the claimed result for $s=1$. We check directly that variables $z_{D,*}$ do not change. For $z_{E,*}$ variables we obtain
\begin{align*}
z_{E,4i+2b+c}(5t+3+s)&= z_{E,4i+2b+c}(5t+3)+(y_{F,4i+2b+c}(5t+3+s)-y_{F,4i+2b+c}(5t+3)) \\
&=2+\Upsilon(i,b,c)+\mb{1}\{(i,b,c)=(i^t,b^t,c^t)\}\\
&+4(s-(3+\Upsilon(i,b,c)+\mb{1}\{(i,b,c)=(i^t,b^t,c^t)\})/4)^+.
\end{align*}
Then for $s=1$ we obtain $z_{E,4i+2b+c}(5t+4)=3$ as claimed. Finally, we have already verified that $z_{F,*}(5t+4)=0$.
This completes the proof of the lemma.
\end{proof}

\subsection{Interval $[5t+4,5t+5]$}

\begin{lemma}\label{lemma:5t+4 5t+5 positive counters}
In the time period $(5t+4,5t+5)$ variable $y_{A,5}$ is active and  all other variables are passive.
The state $z(5t+5)$ is as claimed in  Theorem~\ref{theorem:counter-to-Skorokhod}.
\end{lemma}

\begin{proof}
It is straightforward to verify starting from $y_{B,*}$ till $y_{F,*}$, in this order that all of these variables remain passive. Thus
$y_{A,5}$ is the only active variable. This variable impacts only variables $z_{F,*}$. Specifically, for $s\in (0,1)$
\begin{align*}
z_{F,4i+2b+c}(5t+4+s)=z_{F,4i+2b+c}(5t+4)+4(y_{A,5}(5t+4+s)-y_{A,5}(5t+4))=4s,
\end{align*}
where Lemma~\ref{lemma:5t+3 5t+4 positive counters} was used to assert $z_{F,4i+2b+c}(5t+4)=0$. Fixing $s=1$,
we obtain the claim.
\end{proof}

Lemmas~\ref{lemma:5t 5t+1 positive counters}--\ref{lemma:5t+4 5t+5 positive counters}
are combined to complete the proof
of Theorem~\ref{theorem:counter-to-Skorokhod}.
\end{proof}

\section{Modification of a Skorokhod problem and stability}\label{section:wrapup}
In this section we modify the Skorokhod problem in such a way that the Counter Machine enters a halting state $i^*$
if and only if the Skorokhod problem initialized as described in the beginning of Section~\ref{section:dynamics}
is stable. Additionally we show that the $R$ matrix corresponding to the modified problem is completely-$\mathcal{S}$.
We then use Theorem~\ref{theorem:counter_undecidable2} to complete the proof of our main result, Theorem~\ref{theorem:mainresult}.

Thus consider the Skorokhod problem constructed in the previous section. Suppose that $i^*$ is the halting state.
Recall by Theorem~\ref{theorem:counter_undecidable2} that we may assume without the loss of generality that if the halting state $i^*$ is reached at some
time $T$, then the values of the counters at this time are given as $C^T_1=C^T_2=1$.
We now modify the Skorokhod problem as follows. We do not change the vector $\theta$.
The matrix $R$ is modified as follows. Set $AB_{k,i^*}=-1$ for $k=3,4,5$ (in particular the matrix $AB$ is no longer zero).
Set $EB_{4i+2b+c,i^*}=-3+b+c$ for all $i=1,\ldots,m$ and $b,c=0,1$ (it used to be $-1$ for $i=i^*$ and zero for $i\ne i^*$).
Finally, let $FB_{4i+2b+c,i^*}=-4+b+c$ for all $i=1,\ldots,m$ and $b,c=0,1$. All the remaining entries remain the same.
For convenience we denote the modified matrix by $R$ again.

\begin{prop}\label{prop:S-completeness}
The matrix $R$ is completely-$\mathcal{S}$.
\end{prop}

\begin{proof}
Consider a vector $v$ defined as follows:
$v_{A,1}=v_{A,2}=1,v_{A,3}=v_{A,4}=v_{A,5}=3,v_{B,i}=2, 1\le i\le m,v_{C,*}=5,v_{D,*}=2,v_{E,*}=25, v_{F,*}=38$.
Consider any principal submatrix $\tilde R$ and let $I$ be the set of indices
in $\tilde R$. We claim that $\tilde R v_I>0$. Namely, we claim that the vector $v$ achieves the required property for \emph{every}
principal submatrix $\tilde R$.
Consider any $(A,k)\in I$ (if any exists). Then
\begin{align*}
\sum_{j\in I}R_{(A,k),j}v_j=\sum_{j\in I,j\ne (A,k)}R_{(A,k),j}v_j+v_{A,k}
\end{align*}
But the only negative component in the row $(A,k)$ is $AB_{k,i^*}=-1$ for $k=3,4,5$. Since $v_{B,i^*}=2$ and $v_{A,k}=3$ for $3\le k\le 5$
we obtain that the expression is at least $3-2>0$.

Consider any $(B,i)\in I$. We have by non-negativity of $BE$
\begin{align*}
\sum_{j\in I}R_{(B,i),j}v_j\ge -v_{A,1}\mb{1}\{(A,1)\in I\}+v_{B,i}\ge -1+2>0.
\end{align*}

Consider any $(C,i)\in I$. We have by non-negativity of $CE$
\begin{align*}
\sum_{j\in I}R_{(C,i),j}v_j\ge -v_{A,1}\mb{1}\{(A,1)\in I\}-v_{A,4}\mb{1}\{(A,4)\in I\}+v_{C,i}\ge -1-3+5>0.
\end{align*}
Similarly for any $(D,i)\in I$
\begin{align*}
\sum_{j\in I}R_{(D,i),j}v_j\ge -v_{A,2}\mb{1}\{(A,2)\in I\}+v_{D,i}\ge -1+2>0.
\end{align*}
For any $(E,4i+2b+c)\in I, i\ne i^*$ we have
\begin{align*}
\sum_{j\in I}R_{(E,4i+2b+c),j}v_j&\ge -(b+c)v_{A,1}\mb{1}\{(A,1)\in I\}-v_{A,3}\mb{1}\{(A,3)\in I\}-v_{B,i}\mb{1}\{(B,i)\in I\}\\
&-(3-b-c)v_{B,i^*}\mb{1}\{(B,i^*)\in I\}
+(2b-1)v_{C,1}\mb{1}\{(C,1)\in I\}\\
&+(2c-1)v_{C,2}\mb{1}\{(C,2)\in I\}+v_{E,4i+2b+c}\\
&\ge -2-3-2-3\times 2-5-5+24\\
&>0.
\end{align*}
The case $i=i^*$ is similar, except the term $v_{B,i}\mb{1}\{(B,i)\in I\}$ disappears.
Similarly, for  any $(F,4i+2b+c)\in I$ we have
\begin{align*}
\sum_{j\in I}&R_{(F,4i+2b+c),j}v_j\\
&\ge -(b+c)v_{A,1}\mb{1}\{(A,1)\in I\}-v_{A,3}\mb{1}\{(A,3)\in I\}-4v_{A,4}\mb{1}\{(A,4)\in I\}-v_{B,i}\mb{1}\{(B,i)\in I\}\\
&-(4-b-c)v_{B,i^*}\mb{1}\{(B,i^*)\in I\}
+(2b-1)v_{C,1}\mb{1}\{(C,1)\in I\}\\
&+(2c-1)v_{C,2}\mb{1}\{(C,2)\in I\}+v_{F,4i+2b+c}\\
&\ge -2-3-4\times 3-2-4\times 2-5-5+38\\
&>0.
\end{align*}
\end{proof}

Next we connect the halting property of the Counter Machine with stability of the Skorokhod problem.
Introduce a new notation. Given a state $i$ and counter values $C_1,C_2$, the corresponding
state of the Skorokhod problem described in the beginning of Section~\ref{section:dynamics} is
denoted by $z^{i,C_1,C_2}$. In particular $z^{i,C_1,C_2}_{B,i}=0, z^{i,C_1,C_2}_{B,i'}=1, i'\ne i$,
and $z^{i,C_1,C_2}_{C,i}=C_i, ~i=1,2$.
\begin{prop}\label{prop:stability}
Consider a Counter Machine starting in configuration $(i^0,z^0_1,z^0_2)$ and the corresponding
Skorokhod problem $(\theta,R)$ initiated
in state $z(0)=z^{i^0,z^0_1,z^0_2}$. Then the Counter Machine enters state $i^*$ if and only if the Skorokhod problem $(z(0),\theta,R)$
is stable. Specifically, if the Counter Machine enters state $i^*$ in time $T$, then $z(5T+1)=0$, and if the Counter Machine
never reaches state $i^*$ then $z(t)\ne 0$ for all $t$.
\end{prop}

\begin{proof}
Assume first that the halting state $i^*$ is never reached, and let us show that
$(z(0),\theta,R)$ is not stable.
For this goal let us  show that the dynamics of $y(t),z(t)$ evolve exactly in the same way
as in the case before augmentation. The proof is by induction in $t$. The base case is covered
by the initialization of our Skorokhod problem. Suppose the assertion holds for all $t'\le 5t$. In particular, $z(5t)=z^{i^t,C^t_1,C^t_2}$.
We now show that it holds over time
$[5t,5t+5]$. In particular, $z(5t+5)=z^{i^{t+1},C^{t+1}_1,C^{t+1}_2}$.
Consider first the interval $[5t,5t+1]$. Observe that $y_{A,1}$ increases at most the unit rate over the interval $(5t,5t+1)$,
since all the entries in the row $(A,1)$ are non-negative.
Since $i^*\ne i^t$ then $z_{B,i^*}(5t)=1$. Also $z_{B,i^*}$ can decrease at the rate at most the rate of increase of $y_{A,1}$ (as all other entries
in the row $(B,i^*)$ of $R$ are non-negative. It follows that $z_{B,i^*}$ remains positive in the interval $(5t,5t+1)$ and therefore,
$y_{B,i^*}$ is passive. Since the only change in the matrix $R$ has to do with the column variable $(B,i^*)$, then repeating the
argument of Lemma~\ref{lemma:xA dynamics} and Lemma~\ref{lemma:5t 5t+1 positive counters} we obtain that the dynamics of $y(t),z(t)$
is the same as before augmentation. In particular the state $z(5t+1)$ is described in Lemma~\ref{lemma:5t 5t+1 positive counters}.

Now consider interval $(5t+1,5t+2)$. First consider $(5t+1,5t+3/2)$. Recall that $y_{A,1}$ increases at most the unit rate.
It then follows that $y_{B,i^*}$  also can increase
at most the unit rate since $BA_{i^*,1}=-1$ and all the other entries in the row $(B,i^*)$ are non-negative.
Let us show that over this interval $y_{A,k}$ are passive for $k=3,4,5$.
Their value at time $5t+1$ is $1$. They can decrease at most the rate $-2$, with $-1$ due to $\theta_{A,k}=-1$ and another $-1$ due
to the increase rate of $y_{B,i^*}$ being bounded by one. Thus $y_{A,k}, k=3,4,5$ are passive in $(5t+1,5t+3/2)$. This further implies
that $y_{A,2}$ increases at the unit rate over this interval (to compensate $\theta_{A,2}=-1$). But then as in the proof of Lemma~\ref{lemma:xA dynamics}
we conclude that $y_{A,1}$ remains passive. This further implies that $y_{B,i^*}$ remains passive. We conclude that over the time
interval $(5t+1,5t+3/2)$ the system evolves exactly in the same way as the system before modification,
namely as described in Lemma~\ref{lemma:5t+1 5t+2 positive counters}. In particular $z_{A,k}=1$ for $k=3,4,5$. We now repeat the same argument as
for the interval $(5t+1,5t+3/2)$ to argue that over the time interval $(5t+3/2,5t+2)$ the dynamics is the same as before the modification.

Now consider interval $(5t+2,5t+3)$. Since $z_{A,1}(5t+2)=1$, then $y_{A,1}$ remains passive in this time interval. Then $y_{B,i^*}$ remains
passive in this time interval as well. Then the dynamics of $y$ and $z$ is the same as before the augmentation and is described in
Lemma~\ref{lemma:5t+2 5t+3 positive counters}. In particular $z_{B,i^*}(5t+3)=1$ (since $i^*\ne i^t$). Recall from
Lemmas~\ref{lemma:5t+3 5t+4 positive counters} and \ref{lemma:5t+4 5t+5 positive counters} that $z_{B,i}(s)=1$ for $s=(5t+3,5t+5)$ and $i\ne i^{t+1}$.
By assumption $i^*\ne i^{t+1}$. Therefore this applies to $z_{B,i^*}$ as well and in particular the dynamics over $(5t+3,5t+5)$
is exactly as before the augmentation. This completes the proof Proposition~\ref{prop:stability} for the case when the state $i^*$ is never
reached.

Now assume that the halting state $i^*$ is reached in time $t=T$.
We recall from Theorem~\ref{theorem:counter_undecidable2} that then at time $T$ both counters are equal to unity. Repeating the previous argument we have
that $z(5T)=z^{i^*,1,1}$. The claim of Proposition~\ref{prop:stability} follows immediately from the following lemma.

\begin{lemma}\label{lemma:wrapup}
Over the time interval $(5T,5T+1)$ the variables  $y_{A,1}$ and $y_{B,i^*}$ are active at the unit rate. All the remaining variables are passive.
Moreover $z(5T+1)=0$.
\end{lemma}

\begin{proof}
We have
\begin{align*}
z_{A,2}(5T+s)\ge z_{A,2}(5T)-s=1-s.
\end{align*}
It follows that $y_{A,2}$ is a passive variable in $(5T,5T+1)$.

For variables $y_{B,*}, z_{B,*}$ we have for $i\ne i^*$ and by non-negativity of $BE$ sub-matrix that
\begin{align*}
z_{B,i}(5T+s)&\ge z_{B,i}(5T)-(y_{A,1}(5T+s)-y_{A,1}(5T))=1-(y_{A,1}(5T+s)-y_{A,1}(5T))
\end{align*}
Recall that $y_{A,1}$ increases at most the unit rate. This means $z_{B,i}(5T+s)>0$ for all $s\in (0,1)$ and therefore $y_{B,i}$ remains passive.
We will show later that $y_{B,i^*}$ is active. For now let us obtain a partial relation on this variable. We have
\begin{align*}
z_{B,i^*}(5T+s)&= -(y_{A,1}(5T+s)-y_{A,1}(5T))+y_{B,i^*}(5T+s)-y_{B,i^*}(5T)\\
&+\sum_jBE_{i^*,j}(y_{E,j}(5T+s)-y_{E,j}(5T)).
\end{align*}
From the non-negativity of $BE$ submatrix and applying constraint (\ref{eq:non-idling}) we obtain that
\begin{align}\label{eq:A1<i*}
y_{B,i^*}(5T+s)-y_{B,i^*}(5T)\le y_{A,1}(5T+s)-y_{A,1}(5T),
\end{align}
for every $s\in (0,1)$ (otherwise there exist an interval over which $y_{B,i^*}$ is active while $z_{B,i^*}$ is positive).
In particular,  $y_{B,i^*}$ increases at most the unit rate since so does $y_{A,1}$.

Now consider variables $y_{A,3},z_{A,3}$. We have applying (\ref{eq:A1<i*})
\begin{align*}
z_{A,3}(5T+s)\ge 1-s+y_{A,1}(5T+s)-y_{A,1}(5T)-(y_{B,i^*}(5T+s)-y_{B,i^*}(5T))\ge 1-s>0
\end{align*}
for $s\in (0,1)$.
It follows that $y_{A,3}$ is passive in $(5T,5T+1)$. For a similar reason $y_{A,4}$ remains passive as well.
We will return to $y_{A,5}$ later.

Recall that $z^{T}_1=z^{T}_2=1$. This means $z_{C,k}(5T)=1, k=1,2$. Using a similar argument as for $y_{B,i}, i\ne i^*$,
non-negativity of $CE$ and the fact that $y_{A,4}$ is passive, we establish that $y_{C,k}$ remain passive. Variables $y_{D,k}, k=1,2$ remain passive
since $y_{A,2}$ is passive as we established, and all other entries in the $D$ rows of $R$ are non-negative.

For $E$ variables, recall that $y_{A,3},y_{B,i}, i\ne i^*$ and $y_{C,*}$ are passive. Then
\begin{align*}
z_{E,4i+2b+c}(5T+s)&=3-(b+c)(y_{A,1}(5T+s)-y_{A,1}(5T))-(3-b-c)(y_{B,i^*}(5T+s)-y_{B,i^*}(5T))\\
&+y_{E,4i+2b+c}(5T+s)-y_{E,4i+2b+c}(5T)+y_{F,4i+2b+c}(5T+s)-y_{F,4i+2b+c}(5T)
\end{align*}
Recall that $y_{A,1}$ and $y_{B,i^*}$
increase at most the unit rate.  We see
that
\begin{align*}
z_{E,4i+2b+c}\ge 3-(b+c)s-(3-b-c)s=3-3s
\end{align*}
and therefore remains positive in $(5T,5T+s)$. Thus $y_{E,*}$ variables are passive.
For a similar reason and since in addition $y_{A,4}$ is passive, as we have established, variables $y_{F,*}$ are passive as well.

Let let us consider variable $y_{B,i^*}$. In light of the fact that $y_{E,*}$ are passive, we have
\begin{align*}
z_{B,i^*}(5T+s)=-(y_{A,1}(5T+s)-y_{A,1}(5T))+(y_{B,i^*}(5T+s)-y_{B,i^*}(5T)).
\end{align*}
This by constraint (\ref{eq:non-idling}) implies
\begin{align}\label{eq:A1=i*}
y_{B,i^*}(5T+s)-y_{B,i^*}(5T)=y_{A,1}(5T+s)-y_{A,1}(5T)
\end{align}
for all $s\in (0,1)$. Now let us return to $y_{A,1}$.
Since $y_{A,2},y_{A,3},y_{A,4}$ are passive then
\begin{align*}
z_{A,1}(5T+s)&=-s+y_{A,1}(5T+s)-y_{A,1}(5T),
\end{align*}
which implies that $y_{A,1}$ is active at unit rate and the same holds for $y_{B,i^*}$ by (\ref{eq:A1=i*}).

Finally, turning to $y_{A,5}$ we have
\begin{align*}
z_{A,5}(5T+s)&=-s+2(y_{A,1}(5T+s)-y_{A,1}(5T))+y_{A,5}(5T+s)-y_{A,5}(5T)\\
&-(y_{B,i^*}(5T+s)-y_{B,i^*}(5T))\\
&=-s+2s+y_{A,5}(5T+s)-y_{A,5}(5T)-s\\
&=y_{A,5}(5T+s)-y_{A,5}(5T).
\end{align*}
This implies that $y_{A,5}$ is passive in the interval $(5T,5T+1)$.

We have established that $y_{A,1}$ and $y_{B,i^*}$ are the only active variables and they increase at the unit rate.
With this in mind it is straightforward to verify that  $z(5T+1)=1$. This completes the proof of the lemma.
\end{proof}

This completes the proof of Proposition~\ref{prop:stability}.
\end{proof}

\section*{Acknowledgements}
We are delighted to acknowledge many enlightening discussions with Maury Bramson.

\bibliographystyle{amsalpha}

\newcommand{\etalchar}[1]{$^{#1}$}
\providecommand{\bysame}{\leavevmode\hbox to3em{\hrulefill}\thinspace}
\providecommand{\MR}{\relax\ifhmode\unskip\space\fi MR }
\providecommand{\MRhref}[2]{%
  \href{http://www.ams.org/mathscinet-getitem?mr=#1}{#2}
}
\providecommand{\href}[2]{#2}

\end{document}